\documentclass[a4paper,12pt]{article}

\usepackage[utf8]{inputenc}
\usepackage[english]{babel}
\usepackage{bm}
\usepackage{amssymb, amsmath, amsthm, mathrsfs}
\usepackage{siunitx, comment, xcolor}
\usepackage[inline,shortlabels]{enumitem}
\usepackage{changepage}
\usepackage[numbers]{natbib}
\usepackage{times}
\usepackage{titlesec}
\usepackage{tocloft}

\usepackage{fix-cm}
\usepackage{xpatch}
\xpatchcmd\swappedhead{~}{.~}{}{}
\allowdisplaybreaks
\usepackage{perpage} 
\MakePerPage{footnote}

\newtheoremstyle{mythm}
{}                
{}                
{}        
{}                
{\sc}       
{.}               
{0.5em}               
{}                
\theoremstyle{mythm}
\newtheorem{thm}{Theorem}

\newtheorem{cor}[thm]{Corollary}

\newtheorem{ex}[thm]{Example}
\newtheorem{re}{Remark}
\newcommand{\thistheoremname}{}
\newtheorem{genericthm}[thm]{\thistheoremname}

\newcommand{\A}{\mathcal{A}}

\newcommand{\code}{\mathscr{C}}
\newcommand{\bgw}{\text{BGW}}

\newcommand{\oa}{\text{OA}}
\newcommand{\ca}{\text{CA}}
\newcommand{\w}{\text{W}}
\newcommand{\gf}{\text{GF}}

\newcommand{\wt}{\text{wt}}
\newcommand{\dist}{\text{dist}}
\newcommand{\M}{\text{A}}

\newcommand{\sn}{\mathfrak{S}_v}
\newcommand{\fn}{\text{Func}}
\newcommand{\aut}{\text{Aut}}

\renewenvironment{proof}{{\indent\it Proof.}\hspace{0.2em}}{\mbox{
    \rule{0.075in}{0.15in}}} 

\newenvironment{myabstract}{\vspace{1em}\begin{adjustwidth}{3em}{3em}\begin{small}\textsc{Abstract.}\hspace{0.5em}}{\end{small}\end{adjustwidth}\vspace{1em}}

\titleformat{\section}{\normalfont\Large\scshape\centering}{\thesection.}{0.5em}{}
\titleformat{\subsection}{\normalfont\scshape}{\thesubsection.}{0.5em}{}

\begin{document}
\begin{center}
  {\Large\bf Optimal constant weight codes derived from  
  $\omega$-circulant balanced generalized weighing matrices}

    \end{center}
  \vspace{1em}
  HADI KHARAGHANI\footnote{Department of Mathematics and Computer Science,
    University of Lethbridge, Lethbridge AB T1K 3M4, Canada. Email: {\tt
      kharaghani@uleth.ca}}, THOMAS PENDER\footnote{Department of Mathematics,
    University of Simon Fraser, Burnaby BC V5A 1S6, Canada. Email: {\tt
      tsp7@sfu.ca}},
      VLADIMIR D. TONCHEV \footnote{Mathematical Sciences, Michigan Technological University,
     Email: {\tt
      tonchev@mtu.edu}}

\begin{myabstract}
A  family of $\omega$-circulant balanced weighing matrices with classical
parameters is used for the
construction of optimal constant weight codes over an alphabet of size $g+1$
and length $n=(q^m -1)/(q-1)$, where $q$ is an odd prime power, $m>1$, and
$g$ is a divisor of $q-1$.
\end{myabstract}



\section{Introduction}\label{sec-intro}


A {\it weighing matrix} is a square $(-1,0,1)$-matrix $W=\w(v,k)$ of order $v$
such that $WW^t=kI$ for some positive integer $k$. A weighing matrix $W$ is
{\it balanced} if $(|W_{ij}|)$ is the incidence matrix of a symmetric balanced
incomplete block design (see \citet{design-theory} for the relevant definitions). 

The concept of a balanced weighing matrix can be generalized by allowing the
nonzero  elements of the matrix to be taken from some finite abelian group. To
this end, let $G$ be a finite abelian group of order $n$ not containing the
symbol $0$, and let $W$ be a $(0,G)$-matrix of order $v$. We say that $W=\bgw(v,k,\lambda;G)$ is a
{\it balanced generalized weighing matrix} (henceforth BGW) if the following hold: 
 \begin{enumerate*}[(1)]
  \item there are $k>0$ nonzero entries in every row, and
  \item there is a constant $\lambda>n$ such that the multisets
    $S_{ij} = \{W_{i\ell}W_{j\ell}^{-1} : W_{i\ell} \neq 0\text{, } W_{j\ell}
    \neq 0 \text{, } 0 \leq \ell < v\}$, for every $i<j$, contain $\lambda/n$
    instances of each element of $G$. 
 \end{enumerate*}


A matrix $A$ with first row
$(\alpha_0,\dots,\alpha_{n-1})$ is  {\it $\omega$-circulant} if
$A_{ij}=\alpha_{j-i}$ whenever $i \leq j$, and if $A_{ij}=\omega \alpha_{j-i}$
whenever $j<i$, where the indices are calculated modulo $n$. 
The following theorem, due to \citet{jungnickel-tonchev-perfect-codes-2},
gives a simple construction of $\omega$-circulant 
balanced genralized weighing matrices over a finite field $\gf(q)$ of
 order $q$.

\begin{thm}\label{trace-construction}
\label{T1}
 Let $F=\gf(q^{m+1})$, $K=\gf(q)$, let $\beta$ be a primitive element of $F$,
 and let $\omega^{(q^{m+1}-1)/(q-1)}=\beta$. Let
 $u=(Tr_{F/K}(\beta^i))_{i=0}^{k-1}$, and let $W$ be the $\omega$-circulant
 matrix with first row $u$. Then $W$ is a BGW over $\gf(q)^*$ with parameters 
 \begin{equation}\label{classical-parameters}
  \left(
  \frac{q^{m+1}-1}{q-1}, q^m, q^m-q^{m-1}
  \right).
 \end{equation}
\end{thm}

BGWs with parameters of the form (\ref{classical-parameters}) are termed as
BGWs with {\it classical parameters}. 
We note that since
 the matrix in Theorem \ref{T1} is $\omega$-circulant, the entire array is generated by a single
row. 

\begin{ex}
\label{Ex2}
  Let $\gf(5) = \{0,1,\omega,\omega^2,\omega^3\}$.
  Applying the construction of Theorem \ref{T1}, we obtain a $\omega$-circulant matrix (\ref{bgw-ex})
  with parameters
  $\bgw(6,5,4;\gf(5)^*)$.
  
 \begin{equation}\label{bgw-ex}
  \left(\begin{array}{cccccc}
          \omega^{3}&1&\omega^{3}&0&1&1\\
          \omega&\omega^{3}&1&\omega^{3}&0&1\\
          \omega&\omega&\omega^{3}&1&\omega^{3}&0\\
          0&\omega&\omega&\omega^{3}&1&\omega^{3}\\
          1&0&\omega&\omega&\omega^{3}&1\\
          \omega&1&0&\omega&\omega&\omega^{3}\\
  \end{array}\right).
 \end{equation}
\end{ex}

For further properties and applications of balanced generalized weighing matrices we refer the reader to
 \citet{jungnickel-kharaghani-bgw}, 
 \citet{combinatorics-of-symmetric-designs},
and \citet{seberry-od-2017}.

If $\A$ is a finite alphabet of $q$ symbols containing the symbol $0$, then a
{\it code} of length $n$ is a nonempty subset of $\A^n$. 
The {\it weight}  $\wt(x)$ of a codeword $x=x_0 \cdots  x_{n-1}$ is 
the numbner of its nonzero components.
The {\it minimum weight} of a code is defined as the smallest nonzero weight of a codeword.
The {\it Hamming distance} between codewords $x$ and $y$ is
given by $\dist(x,y)=\#\{i : x_i \neq y_i\}$, and the {\it minimum distance} of
the code is
$\dist(\code)=\min_{\begin{smallmatrix}x,y \in \code \\ x \neq
                      y\end{smallmatrix}}\dist(x,y)$.
                  If $\#\code=M$ and $d=\dist(\code)$, then one writes $\code$
                  is an
                  $(n,M,d)_q$-code. 
                  
 A code is {\it equidistant} if there is a constant $d$ such that
 $\dist(x,y)=d$, for all distinct $x,y \in \code$. A code $\code$ is {\it
   constant weight} if there is a constant $w$ such that $\wt(x)=w$, for
 all $x \in \code$.
 To
emphasize the property of being constant weight, we write $\code$ is an
$(n,M,d,w)_q$-code.

Given the parameters $n$, $d$, $q$ and $w$, it is a
fundamental problem to maximize the  code size $M$. 
We denote by  $\M_q(n,d)$ (resp. $\M_q(n,d,w)$) the maximum size of a code
(resp. constant weight code) of length $n$ over and alphabet of size $q$.

Usually, one uses combinatorial arguments to derive an upper bound for
$\M_q(n,d)$. A lower bound $M \leqq \M_q(n,d)$ is exhibited whenever one
actually constructs an $(n,M,d)_q$-code. In the event these  upper and
lower bounds coincide, we have found $\M_q(n,d)$ and the resulting code is
said to be {\it optimal}. 

The following  upper  bounds will be most useful to us (see \citet{pless-book}
and \citet{tonchev-combinatorial-configurations} for details). 
\begin{thm}\label{thm-johnson-bound-1}
(\emph{Restricted Johnson Bound}):
\begin{equation}\label{restricted-johnson-bound}
 \M_q(n,d,w) \leqq \left\lfloor \frac{nd(q-1)}{qw^2-2(q-1)nw+nd(q-1)} \right\rfloor
\end{equation}
provided $qw^2-2(q-1)nw+nd(q-1)>0$.
\end{thm}

\begin{thm}\label{thm-johnson-bound-2}
 (\emph{Unrestricted Johnson Bound}):
  \begin{equation}\label{restricted-johnson-bound-2}
    \M_q(n,d,w) \leqq \left\lfloor \frac{(q-1)n}{w}\M_q(n-1,d,w-1) \right\rfloor.
  \end{equation}
\end{thm}

In Section \ref{S2}, we give the construction of optimal codes meeting the
bounds (\ref{restricted-johnson-bound}) and (\ref{restricted-johnson-bound-2}).

\section{Optimal Codes From Classical BGWs}\label{subsec-bgw-construction} 
\label{S2}

We begin with an example before tending to the main result of the section.

\begin{ex}\label{5-ary-ex}
  Consider again the BGW from  Example \ref{Ex2}, and call this matrix $W$. We form the
  matrices $\omega^iW$, for $i \in \{0,1,2,3\}$ and append the rows of these
  matrices to $W$ in order to form the following code (transposed).
  \[
    \arraycolsep=1.75pt\def\arraystretch{1.0}
    \begin{array}{cccccccccccccccccccccccc}
      \omega^{3}&\omega&\omega&0&1&\omega&1&\omega^{2}&\omega^{2}&0&\omega&\omega^{2}&\omega&\omega^{3}&\omega^{3}&0&\omega^{2}&\omega^{3}&\omega^{2}&1&1&0&\omega^{3}&1\\
      1&\omega^{3}&\omega&\omega&0&1&\omega&1&\omega^{2}&\omega^{2}&0&\omega&\omega^{2}&\omega&\omega^{3}&\omega^{3}&0&\omega^{2}&\omega^{3}&\omega^{2}&1&1&0&\omega^{3}\\
      \omega^{3}&1&\omega^{3}&\omega&\omega&0&1&\omega&1&\omega^{2}&\omega^{2}&0&\omega&\omega^{2}&\omega&\omega^{3}&\omega^{3}&0&\omega^{2}&\omega^{3}&\omega^{2}&1&1&0\\
      0&\omega^{3}&1&\omega^{3}&\omega&\omega&0&1&\omega&1&\omega^{2}&\omega^{2}&0&\omega&\omega^{2}&\omega&\omega^{3}&\omega^{3}&0&\omega^{2}&\omega^{3}&\omega^{2}&1&1\\
      1&0&\omega^{3}&1&\omega^{3}&\omega&\omega&0&1&\omega&1&\omega^{2}&\omega^{2}&0&\omega&\omega^{2}&\omega&\omega^{3}&\omega^{3}&0&\omega^{2}&\omega^{3}&\omega^{2}&1\\
      1&1&0&\omega^{3}&1&\omega^{3}&\omega&\omega&0&1&\omega&1&\omega^{2}&\omega^{2}&0&\omega&\omega^{2}&\omega&\omega^{3}&\omega^{3}&0&\omega^{2}&\omega^{3}&\omega^{2}\\
    \end{array}
  \]
  The entire code is generated by a single
  codeword. In particular, if $(c_0,c_1,c_2,c_3,c_4,c_5)$ is any codeword of the
  above code, then $(\omega c_5,c_0,c_1,c_2,c_3,c_4)$ is also a codeword. We
  call this operation an $\omega$-shift, and note that all possible
  $\omega$-shifts of a given codeword exhaust this code.

  Next, as the code is given by $\omega$-shifts of a single codeword, it is
  clearly of constant weight $w=5$. Furthermore, it is equidistant where the
  distance can be verified to be $d=5$, the length of the codewords is $n=6$,
  and the number of codewords is given by $M=24$.

  It follows that 
  \[
    qw^2 - 2(q-1)nw + nd(q-1) = 5 \cdot 5^2 - 2 \cdot 4 \cdot 6 \cdot 5 + 6
    \cdot 5 \cdot 4 = 5 > 0,
  \]
  and hence
  \[
    \left\lfloor \frac{nd(q-1)}{qw^2 - 2(q-1)nw + nd(q-1)} \right\rfloor =
    \frac{120}{5} = 24.
  \]
  It follows, therefore, that the code above is optimal.
\end{ex}


 {Let $G$ be a finite group not containing the symbol
  0. Let $[v]=\{1,\dots,v\}$, and let Func$([v],G)$  be the collection of all
  functions $f : [v] \rightarrow G$. The symmetric group $\sn$ acts on
  Func$([v],G)$ by $f^\pi(i)=f(\pi^{-1}i)$, where 
  $(f^{\pi_1})^{\pi_2}=f^{\pi_2\pi_1}$. In particular, this defines the
  semidirect product $\fn([v],G) \rtimes \sn$ where
  $(f_1,\pi_1)(f_2,\pi_2)=(f_1f_2^{\pi_1},\pi_1\pi_2)$ for all pairs
  $(f_1,\pi_1),(f_2,\pi_2) \in \fn([v],G) \rtimes \sn$.} 

 {The group $\fn([v],G) \rtimes \sn$ acts  on the $(0,G)$-vectors
  of length $v$ by $(f,\pi)(g_i)=(f(i)g_{\pi^{-1}i})$ where $g \cdot 0 = 0 \cdot
  g = 0$ for every $g \in G$. We can extend the action of $\fn([v],G) \rtimes
  \sn$ to include the
  action of $\aut(G)$ by looking to $\fn([v],G) \rtimes (\aut(G) \times \sn)$ where
  $(f_1,\alpha_1,\pi_1)(f_2,\alpha_2,\pi_2)=(f_1f_2^{(\pi_1,\alpha_1)},\alpha_1\alpha_2,\pi_1\pi_2)$,
  with $f^{(\pi,\alpha)}(i)=[f(\pi^{-1}i)]^\alpha$, for every pair
  of triples $(f,\alpha,\pi),(h,\beta,\sigma) \in \fn([v],G) \rtimes (\aut(G)
  \times \sn)$. Note that we understand
  $f^{(\pi_2\pi_1,\alpha_2,\alpha_1)}=(f^{(\pi_1,\alpha_1)})^{(\pi_2,\alpha_2)}$. 
  The action then becomes $(f,\alpha,\pi)(g_i) =
  (f(i)g_{\pi^{-1}i}^{\alpha^{-1}})$.}

 {We then extend this action to the collection of $v \times v$
  $(0,G)$-matrices by taking
  $(f,h,\alpha,\pi,\sigma)(g_{ij})=(f(i)g_{\pi^{-1}i,\sigma^{-1}j}^{\alpha^{-1}}h(j))$
  for $(f,h,\alpha,\pi,\sigma) \in (\fn([v],G) \times \fn([v],G)) \rtimes
  (\aut(G) \times \sn \times \sn)$, where
  \[
    (f_1,h_1,\alpha_1,\pi_1,\sigma_1)
    (f_2,h_2,\alpha_2,\pi_2,\sigma_2)=
    (f_1f_2^{(\pi_1,\alpha_1)},h_1h_2^{(\sigma_1,\alpha_1)},\alpha_1\alpha_2,\pi_1\pi_2,\sigma_1\sigma_2).
  \]
  This amounts to having $\fn([v],G) \rtimes (\aut(G) \times \sn)$ act on the
  rows of matrix and $\fn([v],G) \rtimes \sn$ act on its columns.\footnote{We
    can realize the action of this group concretely by first applying a group
    automorphism to the entries of the matrix, and then pre- and
    post-multiplying by a pair of monomial matrices.}
}

 {It is clear that if an orbit under this action contains a
$\bgw(v,k,\lambda;G)$, then every matrix in this orbit is also a BGW with the
same parameters. This fact was used to great effect by \citet{mathon-signings}
when the authors enumerated equivalence classes of small parameter BGWs.}

 {Our purpose for introducing equivalence of BGW matrices is the following.
  Every orbit under the action of $(\fn([v],G) \times \fn([v],G)) \rtimes
  (\aut(G) \times \sn \times \sn)$ that contains a BGW also
  contains a BGW with same parameters of the form}
\[
   {
    \begin{pmatrix}
      \mathbf{0} & R \\ \mathbf{1} & D
    \end{pmatrix}.
  }
\]
 {Such a BGW is said to be in normal form. Note that $R$ is a {\it
    generalized Bhaskar Rao design} (a GBRD for short), and $D$ is a $c$-GBRD
  (see \citet{handbook} for definitions).
  We say that $R$ and $D$ are a {\it residual-} and a {\it derived-part},
  respectively, of the normalized BGW. This equivalence is useful and has been
  recently used by the current authors \citep{new-bw,w-mat-construct} in
  order to construct new families of (balanced) weighing matrices. Outside of
  these and the current applications, this structural result appears not to have
  been utilized before.}

 {Finally, we note that if $W$ is a $\bgw(v,k,\lambda; G)$, and if $\alpha
  : G \rightarrow H$ is a group epimorphism, then $(W_{ij}^\alpha)$ is a
  $\bgw(v,k,\lambda; H)$ where we extend $\alpha$ to $G \cup \{0\}$ by defining
  $0^\alpha=0$.}


\begin{thm}\label{thm-main}
  Let $q$ be an odd prime power and $g$ a divisor of $q-1$. For every $m>1$, the
  following holds:
  \begin{align*}
    \M_{g+1}\left( \frac{q^{m+1}-1}{q-1}-1,d,q^m-1 \right) &= q^m, \quad\text{and}\\
    \M_{g+1}\left( \frac{q^{m+1}-1}{q-1},d,q^m \right) &= g\frac{q^{m+1}-1}{q-1},
  \end{align*}
  where
  \[
    d=2q^m-\frac{(g+1)(q^m-q^{m-1})}{g}.
  \]
  Moreover, the optimal code meeting the second bound 
  is generated entirely by the $\omega^{\frac{q-1}{g}}$-shifts of a single codeword.
\end{thm}

\begin{proof}
  By Theorem \ref{trace-construction}, for every prime power $q$ and integer
  $m>1$, there is an $\omega$-circulant $\bgw$, say $W$, with parameters
  (\ref{classical-parameters}) over $\gf(q)^*$. From the preceding discussion,
  there is an equivalent BGW of the form
  \[
    \begin{pmatrix}
      \mathbf{0} & R \\ \mathbf{1} & D
    \end{pmatrix}.
  \]
Let $C$ be a subgroup of $\gf(q)^*$ of order $g$. The mapping $\omega \mapsto
\omega^{\frac{q-1}{g}}$ induces an epimorphism $\alpha : \gf(q)^* \rightarrow C$.
We now show that the rows of the $q^m\times (\frac{q^{m+1}-1}{q-1}-1)$ matrix
$(D_{ij}^\alpha)$, with $0^\alpha=0$, form the desired $(g+1)$-ary optimal code
meeting the first in the statement of the theorem.

That the distance is given as stated is an easy calculation. The restriction of
Johnson's bound is then seen to be 
\[
  (qg-q+g+1)\frac{q^m-1}{q-1}.
\]
Since $qg-q+g+1=(g-1)(q+1)+2>0$, we can apply the restricted Johnson bound
(\ref{restricted-johnson-bound}) to obtain
 \[
    \M_{g+1}\left( \frac{q^{m+1}-1}{q-1}-1,d,q^m-1 \right) =q^m,
  \]
  which is the number of rows in $(D_{ij}^\alpha)$.

Applying the unrestricted Johnson bound (\ref{restricted-johnson-bound-2}), one
finds that
 \[
    \M_{g+1}\left( \frac{q^{m+1}-1}{q-1},d,q^m \right) \le \frac{gn}{w} 
    \M_{g+1}\left( \frac{q^{m+1}-1}{q-1}-1,d,q^m-1 \right)
    =g\left(\frac{q^{m+1}-1}{q-1}\right). 
  \]

  Taking $W'=(W_{ij}^\alpha)$, we have at once that $W'$ is a BGW with classical
  parameters over $C$. Taking the rows of the matrices $W'$, $\omega'W'$,
  $\dots$, $(\omega')^{g-1}W'$, where $\omega'=\omega^{\frac{q-1}{g}}$, to be the
  codewords of $\code$, we have have that $\code$ is a code with the required
  parameters meeting the above bound. Moreover, all of the $\omega'$-shifts of
  the first row of $W'$ generates the entire code. This completes the proof.
 \end{proof}
 
 An important Corollary to the main result of the paper follows.

\begin{cor}\label{thm-main2}
  Let $q$ be an odd prime power. For every $m>1$, there is an optimal equidistant,
  constant weight $q$-ary code with the parameters
  \begin{equation}
    \left(
      \frac{q^{m+1}-1}{q-1}, q^{m+1}-1, q^m, q^m
    \right)
  \end{equation}
  over the alphabet $\gf(q)$, whereupon
  \[
    \M_q\left( \frac{q^{m+1}-1}{q-1},q^m,q^m \right) = q^{m+1}-1.
  \]
  Moreover, the codewords can be assumed to be generated entirely by all the possible
  $\omega$-shifts of a single codeword.
\end{cor}
\begin{proof}
Let $g=q-1$ in Theorem \ref{thm-main}. Then $d=q^m$ and $M=q^{m+1}-1$ leading to
the conclusion that
\[
  \M_q\left( \frac{q^{m+1}-1}{q-1},q^m,q^m \right) = q^{m+1}-1,
\]
as desired.
\end{proof}

A generalized version of the main result in \citet{opt-ter} now follows from
Theorem \ref{thm-main}.

\begin{cor}
  Let $q$ be an odd prime power. For every $m>1$, there is an optimal bidistant,
  constant weight ternary code with the parameters
  \begin{equation}
    \left(
      \frac{q^{m+1}-1}{q-1}, 2\frac{q^{m+1}-1}{q-1},\frac{q^{m-1}(q+3)}{2} , q^m
    \right).
  \end{equation}
  Therefore,
  \[
    \M_q\left( \frac{q^{m+1}-1}{q-1},\frac{q^{m-1}(q+3)}{2},q^m \right)=
    2\left(\frac{q^{m+1}-1}{q-1}\right).
  \]
  Moreover, the code can be assumed to be generated entirely by the
  negashifts of a single code word. 
\end{cor}
\begin{proof}
Take $C=\{-1,1\}$ is a proper subgroup of $\gf(q)^*$ since $q$ is an odd prime
power. The result now follows from Theorem \ref{thm-main}.
\end{proof}

\begin{re}
 For the proper subgroup $C$ in Theorem \ref{thm-main}, the minimum distance is
 a function of the size of the subgroup, and the generated codes are bidistant.
 However, for the case where $C=\gf(q)^*$ the code is equidistant. 
\end{re}

\begin{ex}\label{nega-example}
By taking $C_k=\{-1,1\}$ in Example \ref{bgw-ex} and applying Theorem
\ref{thm-main}, we have the optimal ternary code with parameters $(6,12,4,5)$
generated by the negashifts of the first codeword in  
\[
  \begin{matrix}
    - & + & - & 0 & + & + \\ 
    - & - & + & - & 0 & + \\
    - & - & - & + & - & 0 \\ 
    0 & - & - & - & + & - \\
    + & 0 & - & - & - & + \\ 
    - & + & 0 & - & - & - \\
    + & - & + & 0 & - & - \\
    + & + & - & + & 0 & - \\
    + & + & + & - & + & 0 \\
    0 & + & + & + & - & + \\
    - & 0 & + & + & + & - \\
    + & - & 0 & + & + & +
  \end{matrix}.
\]
\end{ex}

\section{Covering Arrays and Mutually Suitable Latin Squares}\label{sec-ca-msls}

{A configuration closely related to error-correcting codes is that of
  orthogonal arrays. Given a $q$-alphabet $\A$, an {\it orthogonal array} is an
  $N \times k$ array with entries from $\A$ such that every element of $\A^t$
  appears some constant number $\lambda$ times in every $N \times t$ subarray.
  We say that such an array is an $\oa_q(N,k,t,\lambda)$. The standard reference
for these objects is the monograph by \citet{sloane-oas}.}

{From the previous section, we know that for any divisor $g$ of
  $q-1$, there is an optimal constant weight code with the parameters}
\[
  {
    \left( \frac{q^{m+1}-1}{q-1}, g\frac{q^{m+1}-1}{q-1},
      2q^m-\frac{(g+1)(q^m-q^{m-1})}{g}, q^m \right)_{g+1}
  }
\]
 {that is generated entirely by the $\omega$-shifts of a single codeword.
  It isn't difficult to see that by taking $g=q-1$ and appending the all zeros
  codeword to this code, the result is an $\oa_q$ with the parameters}
\[
   {
    \left(
      q^{m+1}, \frac{q^{m+1}-1}{q-1}, 2, 1
    \right).
  }
\]

 {In general, however, the codes of the previous section do not always
  yield an OA; rather, we get the following generalization. Continuing to let
  $\A$ be a $q$-alphabet, a {\it covering array} is an $N \times k$ array with
  entries from $\A$ such that every element of $\A^t$ appears at least $\lambda$
  times in every $N \times t$ subarray. We say that such an array is a
  $\ca_q(N,k,t,\lambda)$. The codes of the previous section then give the
  following.} 

 {
\begin{thm}
  Let $q$ be a prime power, and let $g$ be a divisor of $q-1$. For every
  $m>1$, there is a
  \[
      \ca_{g+1}\left( g\frac{q^{m+1}-1}{q-1}+1,\frac{q^{m+1}-1}{q-1},2,1 \right).
  \]
  Moreover, the rows of the array are closed under the operation of
  $\omega$-shifting, and all of the nonzero rows are entirely generated by
  the $\omega$-shifts of a single row.
\end{thm}
}

 {
  \begin{ex}
    Inspecting Example \ref{nega-example} shows that together with the all zeros
    codeword, it is indeed a $\ca_3(12,6,2,1)$.
  \end{ex}
}

 {We recall that a latin square over a $q$-alphabet $\A$ is a $q \times q$
  matrix whose rows and columns are permutations of $\A$. Two latin squares on
  the $q$-alphabet $\A$ are said to be {\it suitable} in the event that in
  the superimposition of one matrix over the other each row contains exactly one
  entry whose abscissa and ordinate coincide. A collection of latin squares
  over a $q$-alphabet are mutually suitable if every pair of distinct squares in
  the collection are suitable. It was shown in \citet{unbiased-real} that
  mutually orthogonal and mutually suitable latin squares are in bijective
  correspondence, hence the number of mutually suitable latin squares on a
  $q$-alphabet is bounded above by $q-1$.} 

 {We can use our construction of $\omega$-circulant $\oa_q(q^2,q+1,2,1)$s
  to produce maximal families of mutually suitable latin squares as expressed in
  the following example.}

 {
  \begin{ex}
    Consider again the $(6,24,5,5)_5$-code constructed in Example
    \ref{5-ary-ex}. Adding the all zeros codeword, and permuting the order of
    the codewords, we obtain the following $\oa_5(25,6,2,1)$ (transposed).
    \begin{small}
      \[
        \arraycolsep=1.75pt\def\arraystretch{1.0}
        \begin{array}{ccccc|ccccc|ccccc|ccccc|ccccc}
          0&0&0&0&0&\omega&\omega&\omega&\omega&\omega&\omega^2&\omega^2&\omega^2&\omega^2&\omega^2&\omega^3&\omega^3&\omega^3&\omega^3&\omega^3&1&1&1&1&1\\ \hline
          0&\omega&\omega^2&\omega^3&1&\omega^3&\omega&1&0&\omega^2&1&\omega^2&\omega&0&\omega^3&1&\omega&\omega^3&\omega^2&0&0&\omega&\omega^2&1&\omega^3\\
          0&\omega&\omega^2&\omega^3&1&1&\omega^3&0&\omega^2&\omega&\omega&1&0&\omega^3&\omega^2&\omega^3&\omega^2&\omega&0&1&\omega&1&\omega^3&\omega^2&0\\
          0&\omega^3&1&\omega&\omega^2&\omega^3&1&\omega&\omega^2&0&1&\omega&\omega^2&\omega^3&0&0&\omega&\omega^2&\omega^3&1&\omega&0&\omega^2&\omega^3&1\\
          0&1&\omega&\omega^2&\omega^3&0&\omega^3&\omega&1&\omega^2&0&1&\omega^2&\omega&\omega^3&1&0&\omega&\omega^3&\omega^2&\omega^3&\omega&0&\omega^2&1\\
          0&\omega^3&1&\omega&\omega^2&1&0&\omega^3&\omega&\omega^2&\omega&0&1&\omega^2&\omega^3&1&\omega^2&0&\omega&\omega^3&1&\omega&\omega^3&0&\omega^2\\
        \end{array}
      \]   
    \end{small}
    If we take the $5 \times 5$ subarrays below the $\omega$s, the $\omega^2$s,
    the $\omega^3$s, and the $1$s of the first row, we obtain the following
    complete system of mutually suitable latin squares.
    \begin{align*}
      &\begin{pmatrix}
         \omega^3&\omega&1&0&\omega^2\\
         1&\omega^3&0&\omega^2&\omega\\
         \omega^3&1&\omega&\omega^2&0\\
         0&\omega^3&\omega&1&\omega^2\\
         1&0&\omega^3&\omega&\omega^2\\
       \end{pmatrix}
      &
      &\begin{pmatrix}
         1&\omega^2&\omega&0&\omega^3\\
         \omega&1&0&\omega^3&\omega^2\\
         1&\omega&\omega^2&\omega^3&0\\
         0&1&\omega^2&\omega&\omega^3\\
         \omega&0&1&\omega^2&\omega^3\\ 
       \end{pmatrix}
      \\
      &\begin{pmatrix}
         1&\omega&\omega^3&\omega^2&0\\
         \omega^3&\omega^2&\omega&0&1\\
         0&\omega&\omega^2&\omega^3&1\\
         1&0&\omega&\omega^3&\omega^2\\
         1&\omega^2&0&\omega&\omega^3\\
       \end{pmatrix}
      &
      &\begin{pmatrix}
         0&\omega&\omega^2&1&\omega^3\\
         \omega&1&\omega^3&\omega^2&0\\
         \omega&0&\omega^2&\omega^3&1\\
         \omega^3&\omega&0&\omega^2&1\\
         1&\omega&\omega^3&0&\omega^2\\
       \end{pmatrix}
    \end{align*}
  \end{ex}
}

 {That the collection of latin squares of the previous example are indeed
  mutually suitable follows at once from the fact that we had extracted them
  from the $\oa_5(25,6,2,1)$ in the way that we had. This construction can be made
  perfectly general.}

 {
  \begin{thm}
    There exists an $\oa_n(n^2,n+1,2,1)$ if and only if there is a complete
    system of mutually suitable latin squares over some $n$-alphabet. If $n$ is
    a prime power, then this system is generated entirely by the $\omega$-shifts
    of a single row and they are directly embedded in the $\oa$.
  \end{thm}
}

\end{document}